\documentclass[12pt]{amsart}

\usepackage{lipsum}
\usepackage{amsfonts}
\usepackage{graphicx}
\usepackage{hyperref,bm,comment} 
\usepackage{subcaption}
\usepackage{epstopdf}
\usepackage{amssymb,latexsym}
\usepackage{cite}
\usepackage{cases}
\usepackage[utf8]{inputenc}
\usepackage{calrsfs}
\usepackage{algorithm}
\usepackage{algpseudocode}
\usepackage{booktabs} 
\usepackage{siunitx} 
\usepackage{caption} 
\usepackage{geometry} 
\usepackage{mathrsfs}  
\DeclareMathAlphabet{\pazocal}{OMS}{zplm}{m}{n}
\ifpdf
  \DeclareGraphicsExtensions{.eps,.pdf,.png,.jpg}
\else
  \DeclareGraphicsExtensions{.eps}
\fi

\usepackage{enumitem}
\setlist[enumerate]{leftmargin=.5in}
\setlist[itemize]{leftmargin=.5in}

\newcommand{\mG}{\mathcal{G}}

\newcommand{\rank}{\mbox{rank}}

\newcommand{\R}{\mathbb{R}}

\newcommand{\N}{\mathbb{N}}

\newcommand{\mb}{\begin{bmatrix}}
\newcommand{\me}{\end{bmatrix}}
\newtheorem{thm}{Theorem}[section]

\newtheorem*{re}{Remark}

\DeclareMathOperator{\E}{\mathbb{E}}
\DeclareSymbolFont{largesymbol}{OMX}{yhex}{m}{n}
\DeclareMathAccent{\Widehat}{\mathord}{largesymbol}{"62}
\numberwithin{equation}{section}

\title{Estimation of Cointegration Vectors in Time Series via Global Optimisation}
\author{Alvey Qianli Lin}
\address{Department of Mathematics, University of Hong Kong, Pokfulam, Hong Kong}
\email{u3594734@connect.hku.hk}
\author{Zhiwen Zhang}
\address{Department of Mathematics, University of Hong Kong, Pokfulam, Hong Kong}
\email{zhangzw@hku.hk}

\date{\today}

\usepackage{amsopn}

\DeclareMathOperator*{\argmin}{arg\,min}

\ifpdf
\hypersetup{
  pdftitle={Cointegration Test in Time Series Analysis by Global Optimisation},
  pdfauthor={A. Q. Lin}
}
\fi

\begin{document}

\maketitle

\begin{abstract}
	Time Series Analysis has been given a great amount of study in which many useful tests were developed. The phenomenal work of Engle and Granger \cite{engle1987co} in 1987 and Johansen \cite{johansen1988statistical} in 1988 has paved the way for the most commonly used cointegration tests so far. Even though cointegrating relationships focus on long-term behaviour and correlation of multiple nonstationary time series, oftentimes we encounter statistical data with limited sample sizes and other information. Thus other tests with empirical advantages may also be of considerable importance. In this paper, we provide an optimisation approach motivated by the Blind Source Separation, or also known as Independent Component Analysis, for cointegration between financial time series. Two methods for cointegration tests are introduced, namely decorrelation for the bivariate case and maximisation of nongaussianity for higher-dimensions. We highlight the empirical preponderances of independent components and also the computational simplicity, compared to common practices of cointegration such as the Johansen's Cointegration Test \cite{johansen1991estimation}. The advantages of our methods, especially the better performances in limited sample size, enable a wider range of application and accessibility for researchers and practitioners to identify cointegrating relationships.
	
\end{abstract}
\maketitle
\section{Introduction}

In the realm of stochastic processes, time series analysis has long been a subject of study due to its wide application in forecasting across various fields. The challenges posed by time series analysis are often attributed to the notion of nonstationarity, characterized by unpredictable fluctuations in data. Engle and Granger in 1987 \cite{engle1987co} \cite{engle1987forecasting} proposed a cointegration process to first test for the presence of a unit root, and then estimate the cointegrating relationship which involves regressing one non-stationary process on another, and the residuals indicate the presence of cointegration. Johansen's work in 1988 \cite{johansen1988statistical} \cite{johansen1991estimation} was built upon Engle and Granger's approach, while introducing a more general framework which accommodates high-dimensional cointegration and an estimation of the cointegration rank and the corresponding vectors. Both of the former works have been cited extensively in the literature, such as in the use of cointegration in exchange rate or interest rate markets \cite{macdonald1992stable} or macroeconomic elements \cite{cheung1997further}. Later work also includes the use of panel method \cite{pedroni1999critical} and wavelet analysis \cite{nason1999wavelets}. Some of the modern researches have addressed the potential drawback for the implementation of Johansen's technique, which is the lack of precision with the presence of small sample size data. This motivates improvements and developments on the original methodology. One aspect which has been constantly avoided while not exploited sufficiently is the nongaussian nature of the stochastic processes, which is closely related to the method of this paper. The limited number of observations typically available in macroeconomic time series data means that the nongaussian distribution of the underlying data can be informative when testing for cointegrating relationships.

In this paper, we introduce two novel methods based on global optimisation inspired by Blind Source Separation (\cite{belouchrani1997blind}, \cite{xin2014mathematical}) in view of lower and higher dimensions for practical analysis. Blind source separation is widely used in signal processing and speech modeling by its own nature to separate out the desired signals through a noised mixture of signals, akin to our brain's ability to process voices and sounds in a clamorous environment. The idea here is that we may view several time series as mixtures, then we devise a system such that they are mixtures of some other processes that we are familiar with and are easy to process. A blind source separation is done without any primary knowledge of the mixing information. Hence, by such a method the separation from the original series becomes promising. Our global optimisation also gives an estimate for the coefficients used in cointegration, which plays a significant role in convergence trading. The empirical advantage of the implementation of Independent Component Analysis is that when evaluating the cointegration vector, the orders of cointegration of the time series are not required as a prerequisite of running the algorithm. Also notice that when dealing with new datasets or specific time periods or events, sample sizes of statistical data could be small and highly non-normal, while this paper would show how the investigation of independent components provides a better performance in simulating the long-term trends of time series.

The structure of this paper is as follows, section 2 introduces preliminary knowledge about multivariate time series and the basic algorithms behind blind source separation, specifically the methodology of decorrelation and maximization of nongaussianity. \cite{hyvarinen2000independent}. Furthermore, the Johansen's cointegration test would also be introduced for further reference and comparison with our own algorithm (see Appendix). In section 3, we discuss the theoretical background and convergence properties of the methods. Moving on to section 4, the paper provides practical applications of the methods using both simulated series and oil prices as a real-life example. Section 4 also brings out the empirical advantages and drawbacks with comparisons within our methods and Johansen's method. 

\section{Preliminaries}\label{sec:prelim}

\subsection{Vector Autoregressive Models}

\hfill

Let $\mathbf{s}_t$ be a $k$-dimensional vector, a Vector Auto-Regressive Model (VAR(p)) is defined as:
$$\mathbf{s}_t = \boldsymbol{\phi}_0 +\sum_{i=1}^{p}{\Phi_i\mathbf{s}_{t-i}} + \boldsymbol{\varepsilon}_t$$
where $\mathbf{\phi_0}$ is a $k$-dimensional vector and $\Phi_i$ is a $k \times k$ coefficient matrix for $i = 1, 2, \ldots, p$.

By introducing the Back-Shift Operator, we see that a VAR(p) process can be transformed into
$$\Phi(B)\mathbf{s}_t=\boldsymbol{\phi}_0+\boldsymbol{\varepsilon}_t$$
where $\Phi(B) = I - \Phi_1B -\Phi_2B^2 - \ldots - \Phi_pB^p$.

Now consider $\mathbf{x}_t := (\mathbf{s}_{t-p+1} , \mathbf{s}_{t-p+2}, \ldots, \mathbf{s}_t) \in \R^k \times \R^p$ and $\boldsymbol{\delta}_t := (\mathbf{0}, \mathbf{0}, \ldots, \boldsymbol{\varepsilon}_t) \in \R^k \times \R^p$. A VAR(p) model is equivalent to 
$$\mathbf{x}_t = \Phi_* \mathbf{x}_{t-1} + \boldsymbol{\delta}_t$$
where $\Phi_*$ is a $kp \times kp$ matrix such that 
$$\Phi_* =  \mb \mathbf{0} & I & \mathbf{0} & \mathbf{0} & \ldots & \mathbf{0} \\ \mathbf{0} & \mathbf{0} & I & \mathbf{0} & \ldots & \mathbf{0} \\ \vdots & \vdots & \vdots & \vdots & \ddots & \vdots \\ \mathbf{0} & \mathbf{0} & \mathbf{0} & \mathbf{0} & \ldots & I \\ \Phi_p & \Phi_{p-1} & \Phi_{p-2} & \Phi_{p-3} & \ldots & \Phi_1 \me $$
This matrix is called the companion matrix of the matrix polynomial $\Phi(B)$ \cite{tsay2005analysis}. In this case, the system is transformed into a model of a compact form of a 1-dimensional VAR(p) which can be analysed componentwisely. From this, the process $\mathbf{x_t}$ can be written in a linear relationship with the series $\boldsymbol{\delta}_t$:
$$\mathbf{x}_t - \boldsymbol{\mu} = \boldsymbol{\delta}_t + \sum_{i=1}^{p}{\Phi_{*}^i}\boldsymbol{\delta}_{t-i}$$
Since $\boldsymbol{\delta}_t$ is a white noise series and $\mathbf{x}_t$ has a linear relationship with the past values of the white noise series, the necessary and sufficient condition for a VAR(p) model to be weakly stationary is that the eignvalues of the companion matrix should be less than 1 in modulus.

Let $\mathbf{s}_t$ be a $k$-dimensional vector, a Vector Moving-Average Model (VMA(q)) is defined as:
$$\mathbf{s}_t = \mathbf{c}_0 + \boldsymbol{\varepsilon}_t - \sum_{j=1}^{q}{\Theta_j\boldsymbol{\varepsilon_{t-j}}}$$
where $\mathbf{c}_0$ and $\boldsymbol{\varepsilon}_{t-i}$ are $k$-dimensional vectors for $i = 0, 1, \ldots, q$ and each $\boldsymbol{\varepsilon}_{t-i}$ is a white noise series. Therefore the most genarl case, a Vector Auto-Regressive Moving-Average (ARMA(p,q)) Model can be written as
$$\mathbf{s}_t - \sum_{i=1}^{p}{\Phi_i\mathbf{s}_{t-i}} = \boldsymbol{\phi}_0 + \boldsymbol{\varepsilon}_t - \sum_{j=1}^{q}{\Theta_j\boldsymbol{\varepsilon_{t-j}}}$$
where $\mathbf{s}_t$ is a vector time series, $\boldsymbol{\varepsilon}_t$ is a white noise series, $\Phi_i$ and $\Theta_j$ are coefficient matrices for $i = 1, 2, \ldots, p$ and $j = 1, 2, \ldots, q$ respectively.

\subsection{The blind source separation} \label{sec:bss}

\hfill

The blind source separation, also known as independent component analysis, is a process of separating out independent signals from a mixture of multiple signals. A straightforward system is 
$$\mathbf{x} = A\mathbf{s}$$
where we assume that $\mathbf{x}, \mathbf{s} \in \R^n$ and $A \in \R^{n \times n}$. In a practical problem setting, $\mathbf{x}$ is known in terms of observed data, we then separate the vector $\mathbf{s}$, which is assumed to be consisting of independent components, from the mixing matrix $A$, which is assumed to be nonsingular.

\subsubsection{The decorrelation method} \label{subsec:dec}

\hfill

When $n=2$, assume that the system becomes $\mathbf{x}_t=A\mathbf{s}_t$, i.e.
\begin{equation}\label{eqn:2*2bss}
	\begin{cases}
		x_{1t} = a_{11}s_{1t} + a_{12}s_{2t}, \\
		x_{2t} = a_{21}s_{1t} + a_{22}s_{2t}.
	\end{cases}\,
\end{equation}

We wish to diagonalise \eqref{eqn:2*2bss}, so we set $v_1(t) = a_{22}x_{1t} - a_{12}x_{2t}$ and $v_{2t} = -a_{21}x_{2t} + a_{11}x_{2t}$. Hence, we have
\begin{equation}\label{eqn:2*2diagonalisedsystem}
	\begin{cases}
		v_{1t} = (a_{11}a_{22} - a_{12}a_{21})s_{1t} = det(A)s_{1t}, \\
		v_{2t} = \det(A)s_{2t}.
	\end{cases}\
\end{equation}
If the mixing matrix $A$ is nonsingular, then $v_{1t}$ and $v_{2t}$ are independent as $s_{1t}$ and $s_{2t}$ are independent.

If the transformed signals are stationary for a period of time interval $N \in \R$, which means $\E[v_{1t}v_{2,t-n}] = 0$ for all $n \in [-N,N]$, letting $C_n^{ij} = \E[x_{it}x_{j,t-n}]$, we then have 
\begin{align*}\label{eqn:bss_moment_condition}
	0 & = \E[v_{1t}v_{2,t-n}] \\
	& = \E[(a_{22}x_{1t}-a_{12}x_{2t})(-a_{21}x_{1,t-n}+a_{11}x_{2,t-n})] \\
	& = 	-a_{22}a_{21}C_n^{11}+a_{12}a_{21}C_n^{21}+a_{22}a_{11}C_n^{12}-a_{12}a_{11}C_n^{22}.
\end{align*}
If we further parametrise by letting
\begin{align*}
	& a_{22} = \cos{\theta}, a_{12} = \sin{\theta}, \\
	& a_{21} = \cos{\phi}, a_{11} = \sin{\phi}. 
\end{align*} 
for some angles $\theta$ and $\phi$, we have
$$C_{n}^{11} -  \tan{(\theta)}C_{n}^{21} - \tan{(\phi)}C_{n}^{12} + \tan{(\theta)}\tan{(\phi)}C_{n}^{22} = 0$$
where $C_{n}^{ij} = \E[x_{it}x_{j,t-n}]$.
For $n = 1,2$, rearranging and dividing we get
$$a\tan^2{(\theta)} + b\tan{(\theta)} + c = 0$$
where
\begin{align*}
	& a = C_1^{21}C_2^{22} - C_1^{22}C_2^{21},\\
	& b = C_1^{22}C_2^{11} + C_1^{12}C_2^{21} - C_1^{21}C_2^{12} - C_1^{11}C_2^{22},\\
	& c = C_1^{11}C_2^{12} - C_1^{12}C_2^{11}.
\end{align*} 
Notice that after solving the above quadratic equation, we may get the value for the angle parameters and hence trace back the mixing matrix and the clean signal. The above algorithm is called decorrelation \cite{xin2014mathematical}. Notice that decorrelation essentially transforms the problem into solving a polynomial. 


\subsubsection{Maximisation of nongaussianity} \label{subsec:maxnong}

\hfill

Decorrelation seems straightforward for its primary goal to reduce the problem into solutions of polynomials. However, it is not practical to utilise it when the order gets above three. Sacrificing the accuracy of prediction and recovery is indeed not optimal. However, we have other means to deal with higher order cases.

Consider the following optimisation problem

\begin{equation}\label{nongaussianopt}
	\begin{aligned}
		\hat{\mathbf{w}}=\max_{\mathbf{w}} \quad & \E[\mG (\mathbf{w}^T\mathbf{x})]\\
		\textrm{s.t.} \quad & \mathbf{w}^T\mathbf{w}=1\\
	\end{aligned}
\end{equation}
where the objective value is a measure of nongaussianity \cite{hyvarinen2000independent}, $\mathbf{w}$, $\mathbf{x} \in \R^n$. 

\begin{re}
	Common practices of measuring nongaussianity include the use of functions such as kurtosis, while potential drawbacks of kurtosis is usually related to the sensitivity to the outliers of data. One may visit Hyvärinen et al. (2001) for detailed analysis. A way to make amends, however, is to utilise negentropy in information theory. Detailed explanations are introduced shortly.
\end{re}

The motivation comes from the Central Limit Theorem, since a linear combination of independent random variables produces a more normal distribution. Roughly speaking, one independent-component signal tends to be less stable than any other weighted sum consisting of more independent components of the source signal. Given the original problem setting, $\mathbf{x} = A\mathbf{s}$, we may consider a linear combination of $\mathbf{x}$ given by $y=\mathbf{b}^T\mathbf{x}=\mathbf{b}^TA\mathbf{s}$. Letting $\mathbf{q}=(\mathbf{b}^TA)^T$, if we wish to recover the original signal $\mathbf{s}=A^{-1}\mathbf{x}$ by finding the vector $\mathbf{b}$ so that it is a row of $A^{-1}$, the linear combination $\mathbf{b}^T\mathbf{x}$ would recover one of the independent components \cite{hyvarinen2000independent} so that $\mathbf{q}$ must have exactly one nonzero entry i.e. a unit vector $\mathbf{e_i}$, $i \in \{1,...,n\}$, in which case the linear combination $\mathbf{b}^T\mathbf{x}$ is the least normal, and this explains why we aim to maximise nongaussianity. The remaining components can be recovered by using the condition that $\mathbf{s}$ is assumed to have independent components.

While it remains to construct an appropriate algorithm to arrive at a solution of the optimisation problem. Let us define some key terms for reference. Let $y$ be a random variable, the kurtosis of $y$ is given by
$$kurt(y)=\E[y^4]-3(\E[y^2])^2$$

Essentially kurtosis measures the fourth moments of random variables, and it will be useful in the construction of an approximation of the following quantities as a measure of nongaussianity. However, a general critique of kurtosis is its sensitivity to outliers, meaning that in practical analysis such evaluation may not be stable or robust enough, whereas we do have an amended measure with generally better performance.

Let $\mathbf{y} \in \R^n$ be a random variable with probability density function $p_{\mathbf{y}}$. The (differential) entropy is defined as
$$H(\mathbf{y})=-\int p_{\mathbf{y}}(\boldsymbol{\eta})\log\left(p_{\mathbf{y}}(\boldsymbol{\eta})\right) \  d\boldsymbol{\eta}$$

In statistics, roughly speaking, entropy measures the randomness of a random variable, and here we utilise the fact that a normal random variable produces the largest entropy. With the help of this, we are ready to define the objective function that we aim to use. Notice that we wish to develop a nonnegative measure of nonguassianity which is zero with a gaussian random variable input.

Let $\mathbf{y} \in \R^n$ be a random variable. The negentropy of $\mathbf{y}$ is defined as
$$J(\mathbf{y})=H(\mathbf{y}_g)-H(\mathbf{y})$$
where $\mathbf{y_g}$ is the gaussian random variable of the same correlation matrix as $\mathbf{y}$.

Here we introduce a fixed-point algorithm searching for the solution to the optimisation problem outlined in Hyvärinen et al. (2001), while we leave the statistical analysis and the detail choice of functions to later sections.

We propose the following iterative scheme which gives a solution to \eqref{nongaussianopt}:

\begin{algorithm}
	\caption{Maximisation of nongaussianity in one independent component}
	\begin{algorithmic}[1]
		\State Initialise $\mathbf{w}_0$ of unit norm
		\State $\mathbf{w}_{k+1} \gets \E[\mathbf{x}^Tf(\mathbf{w}_k^T\mathbf{x})]-\E[f'(\mathbf{w}_k^T\mathbf{x})]\mathbf{w}_k$
		\State $\mathbf{w}_k \gets \frac{\mathbf{w}_k}{\lVert \mathbf{w}_k \rVert}$ \Comment{If not converge, go back to 2}
		\State \textbf{return} $\mathbf{w}_k$ 
	\end{algorithmic}
\end{algorithm}

First initialise the vector $\mathbf{w}_0$ of unit norm. Proceed by taking
$$\mathbf{w}_{k+1} = \E[\mathbf{x}^Tf(\mathbf{w}_k^T\mathbf{x})]-\E[f'(\mathbf{w}_k^T\mathbf{x})]\mathbf{w}_k$$ for any $k \in \N_0$ and $f$ is an objective function. Then at each step normalise $\mathbf{w}_k \leftarrow \frac{\mathbf{w}_k}{\lVert \mathbf{w}_k \rVert}$. Then repeat until the increment of $\E[\mathbf{x}^Tf(\mathbf{w}_k^T\mathbf{x})]-\E[f'(\mathbf{w}_k^T\mathbf{x})]\mathbf{w}_k$ is within a preset tolerance level $\delta \in \R$ \cite{hyvarinen2000independent}.

Suppose that we have already chosen an appropriate objective function $\mG \in \pazocal{C}^2$. The choice of a function $f: \R \to \R$ and its derivative $f'$ will be discussed later, yet notice that they are not exactly the direct derivatives of the function $\mG$.

With a recap of the key concepts and tools, we are ready to introduce the whole problem setting and the main algorithm for the application of blind source separation and optimisation on cointegration.

\section{Test cointegration using  global optimisation} \label{sec:optimization}

\subsection{Problem setting} \label{sec:probstn}

\hfill

When dealing with cointegration between time series, one needs to find a stationary linear combination of several time series. Here we adopt the following notation.

Let $(s_{1t}), (s_{2t}), \ldots, (s_{nt})$ be $n$ nonstationary time series. A vector $\boldsymbol{\beta}=(\beta_1,\beta_2, \ldots, \beta_n)^{T}$ is called a cointegration vector if the linear combination of the time series, denoted by $(z_t)$ where $z_t=\sum_{i=1}^{n}{\beta_is_{it}}$ is stationary.

Notice that in the case of cointegration for two variables, the (nontrivial) linear combination which produces a cointegration vector is unique up to a scalar multiplication. Hence, in practical analysis, we usually normalise one fixed entry position to be one. Consider a cointegration vector $\boldsymbol{\beta}=(\beta_1, \beta_2)^T$, we may normalise by setting $\hat{\boldsymbol{\beta}}=(1, \hat{\beta_2})^T$.

\subsection{Global optimisation for one signal and one noise.} \label{sec:opt2*2}

\hfill

Let $s_{1t}$ and $s_{2t}$ be two non-stationary time series. Let $r_t$ be an ARIMA$(p,d,q)$ process. Let $\varepsilon_t$ be a white noise series. Consider the following system
\begin{equation}\label{eqn:2*2CointegrationRepresentation}
	\begin{cases}
		s_{1t} = a_{11}r_t + a_{12}\varepsilon_t, \\
		s_{2t} = a_{21}r_t + a_{22}\varepsilon_t.
	\end{cases}\,
\end{equation}

In other words, $\mathbf{s}_t=A\begin{pmatrix} \mathbf{r}_t \\ \varepsilon_t \end{pmatrix}$. We aim to blindly separate the mixing matrix $A = (a_{ij})_{2 \times 2}$ and obtain the recovered series $r_t$, and then find the cointegration coefficients (i.e. a stationary linear combination $z_t = \beta_1 s_{1t} + \beta_2 s_{2t}$). 

For a 2-dimensional model, we use the decorrelation method \cite{xin2014mathematical}. The problem can be reduced to a quadratic equation
$$a\tan^2{(\theta)} + b\tan{(\theta)} + c = 0$$ where we have the parametrisation $a_{11} = \sin{\phi}, a_{12} = \sin{\theta}, a_{21} = \cos{\phi}, a_{22} = \cos{\theta}$ and $C_{n}^{ij} = \E[s_{it}s_{j,t-n}]$ for $i,n = 1,2$. For the coefficients, we have
\begin{align*}
	& a = C_1^{21}C_2^{22} - C_1^{22}C_2^{21}, \\
	& b = C_1^{22}C_2^{11} + C_1^{12}C_2^{21} - C_1^{21}C_2^{12} - C_1^{11}C_2^{22},\\
	& c = C_1^{11}C_2^{12} - C_1^{12}C_2^{11}.
\end{align*} 

Upon solving the quadratic equations we get the mixing matrix $A$ and hence the recovered signal series $r_t$ and $\epsilon_t$.

For the cointegration vector, consider
\begin{equation*}
	\E[z_t^2]=\E[\big(\beta_1(a_{11}r_t+a_{12}\varepsilon_t)+\beta_2(a_{21}r_t+a_{22}\varepsilon_t)\big)^2].
\end{equation*} 

While $r_t$, $\varepsilon_{1t}$ and $\varepsilon_{2t}$ are mutually noncorrelated. Therefore, we have
\begin{align}
	& \E[z_t^2] =  \{(\beta_1a_{11}+\beta_2a_{21})^2\E[r_t^2]+\beta_1^2a_{12}^2\E[\varepsilon_t^2]+\beta^2a_{22}^2\E[\varepsilon_t^2]\}.
\end{align} 

Since $r_t$ is nonstationary, if we have $\beta_1a_{11}+\beta_2a_{21} = 0$, then
$$z_t = (\beta_1a_{12}+\beta_2a_{22})\varepsilon_t$$

becomes a stationary process.

The criterion $\beta_1a_{11}+\beta_2a_{21} = 0$ can also been seen from

$$A\boldsymbol{\beta} = \mb a_{11}&a_{21}\\ a_{12}&a_{22} \me \mb \beta_1\\ \beta_2 \me = \mb 0 \\ * \me$$
The cointegration vector is determined up to scaling of some component(s). In particular, we may assume without loss of generality that $\beta_2 = 1$, since for the criterion to be satisfied we have to avoid the trivial case where $\beta_1 = \beta_2 = 0$. Now the global optimisation algorithm is carried out by 
\begin{align}\label{eqn:Min4CoIntegratedVector5}
	\argmin_{\beta_1,\beta_2} & \E[(\beta_1s_{1t} + \beta_2s_{2t})^2].  
\end{align}
which can be done easily.

\subsection{Optimisation problems for multiple signals.} \label{sec:optn*n}

\hfill

For a generalised version where one has multiple non-stationary processes, if we have $(n-1)$ mutually uncorrelated nonstationary autoregressive integrated moving average (ARIMA) processes  $r_{1t}, r_{2t},\ldots, r_{n-1,t}$ with an order of integration $I(d)$, $d\geq1$ \cite{tong1990non,tsay2005analysis,taylor2008modelling}. In statistics, the order of integration, denoted $I(d)$, means the minimum number of differences required to obtain a covariance-stationary series. 
Let $\varepsilon_t$ be an independent Gaussian process (stationary series). Suppose there is a nonsingular matrix $A=(a_{ij})_{n\times n}$ such that $\mathbf{s}_t=A\begin{pmatrix} \mathbf{r}_t \\ \varepsilon_t \end{pmatrix}$, where $\mathbf{s}_t=(s_{1t}, \ldots, s_{nt})^T \in \R^n$ and $\mathbf{r}_t=(r_{1t}, \ldots, r_{n-1,t})^T \in \R^{n-1}$. 

To find the cointegration vector  
$\boldsymbol{\beta}=(\beta_1,\beta_2, \ldots, \beta_n)^{T}$ such that $z_t= \sum_{i=1}^{n}{\beta_is_{it}}$ is stationary, we solve the minimisation problem 
\begin{align}\label{eqn:Min4CoIntegratedVector}
	& \argmin_{\boldsymbol{\beta}} \E[z_t^2] \nonumber \\
	= &\argmin_{\boldsymbol{\beta}} \E\left[\left(\sum_{i=1}^n {\beta_i\left(\sum_{j=1}^{n-1}{a_{ij}r_{jt}+a_{in}\varepsilon_{t}}\right)}\right)^2\right] \\
	= &\argmin_{\boldsymbol{\beta}}
	\sum_{j=1}^{n-1}{\left(\sum_{i=1}^{n}{\beta_ia_{ij}}\right)^2\E[r_{jt}^2]} + \sum_{i=1}^{n}{\beta_i^2a_{in}^2\E[\varepsilon_t^2]}.
\end{align} 
since $r_{1t}, r_{2t}, \ldots, r_{n-1,t}$, and $\varepsilon_t$ are mutually uncorrelated. Furthermore, since $r_{1t}, r_{2t}, \ldots, r_{n-1,t}$ are non-stationary, we must have $\sum_{i=1}^{n}{\beta_ia_{ij}}=0$  for any $j = 1, 2, \ldots, n-1$. This gives $z_t = \sum_{i=1}^{n}{\beta_ia_{in}\varepsilon_t}$ which is clearly a stationary process. 

Since $A$ is a nonsingular matrix, there must exist an invertible submatrix by deleting one row and one column from $A$ (note that if all submatrices of a matrix are singular, then the matrix itself has to be singular). By setting the corresponding $\beta_k= 1$, where k is between $1$ and $n$, then
\begin{align}\label{eqn:EquationsFromMinization2}
	\sum_{1 \le i \le n, i \neq k}{\beta_ia_{ij}} = -a_{nj}
\end{align} 
for any $ j = 1, 2, \ldots, n-1$. By the invertibility of the submatrix, such system must have a unique solution.

One may also notice that the least square problem
\begin{equation}
	\argmin_{\beta_1,\beta_2, \ldots, \beta_{n-1}}  \E[(\sum_{i=1}^{n-1}{\beta_is_{it}})^2]  = \nonumber
	\argmin_{\beta_1,\beta_2, \ldots, \beta_{n-1}}  F(\beta_1,\beta_2, \ldots, \beta_{n-1}).
\end{equation} becomes a convex quadratic programming problem, which can be solved easily, and explicitly by solving
$\frac{\partial}{\partial_{\beta_i}}F(\beta_1,\beta_2, \ldots, \beta_{n-1})=0$ for all $i = 1, 2, \ldots, n-1$ as a $(n-1)\times (n-1)$ linear system.

The above is a natural extension of lower-dimensional cases. The underlying problem is, however, the separation of higher order mixtures by blind source separation is not quite the same as decorrelation, as generally decorrelation reduces the system into polynomials of corresponding degree of the system, while it is not practical to search for all the roots and apply source recovery. To deal with higher-order systems, common approaches in source separation include information maximisation method, but such method fail to identify the desired stationary series which is essential for our analysis. Here we adopt maximising nongaussianity \cite{hyvarinen2000independent} as introduced in Section \ref{subsec:maxnong}. The reason is that such method is able to identify gaussianity of signals and mixtures, which are essentially the key subject we intend to separate out from a time series.

Consider the following optimisation problem
\begin{equation}\label{nongaussianopt2}
	\begin{aligned}
		\hat{W}=\max_{W} \quad & \E[\mG(W\mathbf{s}_t)]\\
		\textrm{s.t.} \quad & W^TW=I\\
	\end{aligned}
\end{equation}
where the objective value is a measure of nongaussianity \cite{hyvarinen2000independent} as introduced in \ref{subsec:maxnong}, $W$ is a coefficient matrix, with each row representing a linear combination of the source signal $\mathbf{s}_t$, and $\mathbf{s}_t$ is our observed mixture of time series (objective vector time series for cointegration). Notice that the algorithm recovering one independent component could be used here after orthogonality assumptions to recover all the other independent components. An important note is that the recovered signals might be subject to permutation and scaling, but this is nowhere problematic in our analysis since the essential part is the separated stationary series, and cointegration which naturally comes after the separation.

In order to derive a cointegration vector, we aim to find a linear combination of the components of a series such that it is stationary. From \eqref{nongaussianopt} it only remains to test for stationary of the components of $\hat{W}\mathbf{s}_t$.

The algorithm for maximisation of nongaussianity focuses on one-dimensional properties of a random variable, while the time series variables that we are interested in are of higher dimensions. Here we introduce how orthogonality is used to construct optimisation componentwisely through the random variables, and one may consider below the final algorithm for cointegration in higher dimensions \cite{hyvarinen2000independent}.

\begin{algorithm}
	\caption{Maximisation of nongaussianity in all independent components}
	\begin{algorithmic}[1]
		\State Set $i \gets 1$
		\State Initialise $\mathbf{w}_0^{(i)}$ of unit norm
		\State $\mathbf{w}_{k+1}^{(i)} \gets \E[\mathbf{x}^Tf({\mathbf{w}_k^{(i)}}^T\mathbf{x})]-\E[f'({\mathbf{w}_k^{(i)}}^T\mathbf{x})]\mathbf{w}_k^{(i)}$
		\State $\mathbf{w}^{(i)} \gets \mathbf{w}^{(i)}-\sum_{j=1}^{i-1}\left({\mathbf{w}^{(j)}}^T\mathbf{w}^{(i)}\right)\mathbf{w}^{(j)}$
		\State $\mathbf{w}_k^{(i)} \gets \frac{\mathbf{w}_k^{(i)}}{\lVert \mathbf{w}_k^{(i)} \rVert}$ \Comment{If not converge, go back to 3}
		\State $p \gets p+1$ \Comment{If $p<n$, go back to 2}
		\State \textbf{return} $\hat{W}=\left(\mathbf{w}^{(1)}, \ldots \mathbf{w}^{(n)}\right)^T$ 
	\end{algorithmic}
\end{algorithm}

Starting from one independent component, set an index variable $i=1$. For $i \in \{1, \ldots, n\}$, first initialise the vector $\mathbf{w}_0^{(i)}$ of unit norm, then apply the iterative scheme for recovering one independent component
\begin{equation}\label{mainalgorithm}
	\mathbf{w}_{k+1}^{(i)} = \E[\mathbf{x}^Tf(\mathbf{{w}_k^{(i)}}^T\mathbf{x})]-\E[f'({\mathbf{w}_k^{(i)}}^T\mathbf{x})]\mathbf{w}_k^{(i)}
\end{equation}
for any $k \in \N_0$. The at each step normalise $\mathbf{w}_k^{(i)} \leftarrow \frac{\mathbf{w}_k^{(i)}}{\lVert \mathbf{w}_k^{(i)} \rVert}$. Finally we may get a desired vector $\mathbf{w}^{(i)}$. Then orthogonalise
$$\mathbf{w}^{(i)} \leftarrow \mathbf{w}^{(i)}-\sum_{j=1}^{i-1}\left({\mathbf{w}^{(j)}}^T\mathbf{w}^{(i)}\right)\mathbf{w}^{(j)}$$
Hence, we obtain the coefficient matrix $\hat{W}=\left(\mathbf{w}^{(1)}, \ldots \mathbf{w}^{(n)}\right)^T$.

\subsection{Convergence and stability analysis} \label{sec:analysis}

\hfill

It is important to highlight that in practical applications, computing kurtosis or negentropy as a metric of nongaussianity can be challenging, necessitating the use of reliable and precise approximations to streamline the calculation process. Through the selection of suitable functions $F_i:\R \to \R$ for $i = 1$, $2$, it is possible to derive an effective approximation of negentropy, as demonstrated below:
\begin{equation}\label{approxnegentropy}
	J(\mathbf{w}^T\mathbf{s}) \approx k_1\left(\E[F_1(\mathbf{w}^T\mathbf{s})]\right)^2+k_2\left(\E[F_2(\mathbf{w}^T\mathbf{s})]-\E[F_2(v)]\right)^2
\end{equation}
where $v$ is a standardised (normal) random variable. The constants $k_i$ for $i=1$, $2$ are given by $k_i^2=\frac{1}{2\delta_i^2}$ where
\begin{align*}
	\delta_1^2 & = \int \phi(\xi)\left(F_1(\xi)\right)^2 \  d\xi -\left(\int \phi(\xi)F_1(\xi)\xi \ d\xi\right)^2 \\
	\delta_2^2 & = \int \phi(\xi)(F_2(\xi))^2 \  d\xi - \left(\int \phi(\xi)F_2(\xi) \  d\xi\right)^2 - \frac{1}{2}\left(\int \phi(\xi)F_2(\xi) \  d\xi -\int\phi(\xi)F_2(\xi)\xi^2 \ d\xi\right)^2
\end{align*}
The function $\phi(\xi)=\frac{1}{\sqrt{2\pi}}e^{-\frac{\xi^2}{2}}$ is the standardised gaussian density.

From \eqref{approxnegentropy}, in the case where we only utilise one nonquadratic function (detailed choices will be introduced shortly), we see that $\mG(x) \propto \left(\E[F(x)]-\E[F(v)]\right)^2$. By taking the derivative of $\E[\mG(\mathbf{w}^T\mathbf{s})]$ with respect to $\mathbf{w}$, we have $$\frac{\partial\left(\E[\mG(\mathbf{w}^T\mathbf{s})]\right)}{\partial\mathbf{w}}\propto\left(\E[F(\mathbf{w}^T\mathbf{s})]-\E[F(v)]\right)\E\left[\mathbf{s}f(\mathbf{w}^T\mathbf{s})\right]$$
where $f: \R \to \R$ is the derivative of $F$.

Hence, naturally we can develop a fixed point algorithm by
$$\mathbf{w} \leftarrow \E[\mathbf{s}f(\mathbf{w}^T\mathbf{s})]$$
To alleviate unsatisfactory convergence properties, consider the following operation:
\begin{align}
	\mathbf{w} & = \E[\mathbf{s}f(\mathbf{w}^T\mathbf{s})] \\
	(1+\alpha)\mathbf{w} & = \E[\mathbf{s}f(\mathbf{w}^T\mathbf{s})] + \alpha\mathbf{w}
\end{align}
Since normalisation comes in each step of the iteration, the modification still produces a fixed point iteration with the same fixed points \cite{hyvarinen2000independent}. Hence, it remains to find the value of $\alpha$ which induces a fixed point algorithm with faster convergence properties.

The key idea here is to utilise Newton's method. By our optimisation problem \eqref{nongaussianopt}, the analytical solution is solvable by using Lagrangian:
\begin{equation}\label{lagrangian}
	\E[\mathbf{s}f\left(\mathbf{w}^T\mathbf{s}\right)]+\lambda\mathbf{w}=\mathbf{0}
\end{equation}

Now define $F: \R^n \to \R^n$ by the expression on the left-hand-side, $F(\mathbf{w})=\E[\mathbf{s}f\left(\mathbf{w}^T\mathbf{s}]\right)]+\lambda\mathbf{w}$. We then proceed to solve \eqref{lagrangian} by the Newton's method. After some algebraic manipulations, we obtain
$$\mathbf{w} \leftarrow \E[\mathbf{s}f\left(\mathbf{w}^T\mathbf{s}\right)]-\E[f'\left(\mathbf{w}^T\mathbf{s}\right)]\mathbf{w}$$
which is exactly our fixed point algorithm.

Let $D = \overline B(0,1) \in \R^n$. A continuous function $\mathbf{G}: D \mapsto D$ , $\mathbf{G}=\left(g_1, \ldots, g_n\right)$ has a fixed point in $D$, if there exists $K < 1$ such that 
$$\left\lvert\frac{\partial g_i(x)}{\partial x_j}\right\rvert \le \frac{K}{n}$$
for any $\mathbf{x} \in D$ and for any $j = 1, 2, \ldots , n$. Notice that the above follows from the Brouwer fixed point theorem.

Then $\mathbf{G}$ has a fixed point which can be found by the iterative scheme $\mathbf{x}_{k+1}\leftarrow \mathbf{G}\left(\mathbf{x}_k\right)$ and the sequence converges to $\mathbf{x}^* \in D$ such that 
$$\left\lVert\mathbf{x}_k-\mathbf{x}^*\right\rVert_{\infty} \le \frac{K^k}{1-K}\left\lVert\mathbf{x}_1-\mathbf{x}_0\right\lVert_{\infty}$$

Setting $\mathbf{G}(\mathbf{w})=\E[\mathbf{s}f\left(\mathbf{w}^T\mathbf{s}\right)]-\E[f'\left(\mathbf{w}^T\mathbf{s}\right)]\mathbf{w}$, then $\frac{\partial g_i(\mathbf{w})}{\partial w_i}=-\E\left[f'\left(\mathbf{w}^T\mathbf{s}\right)\right]$. Hence it suffices to choose an objective function with bounded derivatives. Using the algorithm introduced in \ref{subsec:maxnong}, we may get the desired solution for \eqref{nongaussianopt}. Notice that the fixed point algorithm guarantees the order of convergence to be at least quadratic \cite{hyvarinen2000independent}. In fact, for the objective function $F$, we will be able to show that any nonquadratic function can be used in the performance of independent component analysis. 


Finally, for the approximating functions, the following choices are proved to be robust and easy to compute:
\begin{align}
	F^{(1)}(t) & = \frac{1}{a}\log\cosh{at} \label{approxf1}\\
	F^{(2)}(t) & = -e^{-\frac{t^2}{2}} \label{approxf2}
\end{align}
where the constant $a \in (0,1]$. The derivatives are hence $f_1(t)=\frac{dF_1}{dt}=\tanh(at)$, $f_1'(t)=\frac{d^2F_1}{dt^2}=a\left(1-\tanh^2(at)\right)$, $f_2(t)=\frac{dF_2}{dt}=te^{-\frac{t^2}{2}}$, $f_2'(t)=\frac{d^2F_2}{dt^2}=(1-t^2)e^{-\frac{t^2}{2}}$.

The following idea demonstrated in \cite{hyvarinen2000independent} give us the theorems that pave the way for the stability of the algorithm. Notice that the original setting requires stationarity of the source signal and the whitening of the mixed signal. However, it is possible to alleviate such requirements by imposing some other conditions which are easy to be handle in practice \cite{chan2013advantages}.

\begin{thm}\label{stabilitythm1}
	Let $F$ be a smooth even function. Suppose that $\mathbf{s}_t=A\mathbf{r}_t$, $\mathbf{w}^T\mathbf{s}_t=\mathbf{w}^TA\mathbf{r_t}:=\mathbf{q}^T\mathbf{r}_t$ 
	Furthermore, suppose that the following conditions hold:
	\begin{enumerate}
		\item There exists $C$ such that $\E[r_{jt}] \le C$. More precisely, for any data frame $T>0$, there exists $T_0>0$ such that for any $T>T_0$, there exists $c_T \ge c_0 \ge 0$ such that
		$$ \frac{\left\lvert\sum_{t=1}^T r_{it}\right\lvert}{T}\le c_0$$
		\item For any differentiable univariate function $f$ and $g$,
		$$\frac{\left\lvert\sum_{t=1}^Tg(r_{it})f(r_{it})-\sum_{t=1}^Tg(r_{it})\sum_{t=1}^Tf(r_{it})\right\lvert}{T} \le c_T$$
	\end{enumerate}
	Then the local optimum of the value of $\E[F(\mathbf{w}^T\mathbf{s}_t)]$ subject to the condition $\lVert \mathbf{w} \rVert=1$, i.e. $\mathbf{w}^T\mathbf{w}=1$, so that $\mathbf{q}^T\mathbf{q}=\eta$ for some $\eta \in \R$, include those rows of $WA$ such that the corresponding independent components $r_{i}$ satisfy
	\[
	\E\left[\sum_{j\neq i}f'(\eta r_{it})r_{jt}^2-f(\eta r_{it})r_{it}\right]
	\begin{cases}
		>0 & \text{for maxima} \\
		<0 & \text{for minima}
	\end{cases}
	\]
	where $f$ and $f'$ are respectively the first and second derivative of the objective function $F$ used for the approximation of negentropy $\mG$ \cite{chan2013advantages}.
\end{thm}

\begin{proof}
	Denote $H(\mathbf{w}^T\mathbf{s}_t)=H(\mathbf{w}^TA\mathbf{r}_t)=H(\mathbf{q}^T\mathbf{r}_t):=\E\left[F(\mathbf{w}^T\mathbf{s}_t)\right]=\E[F(\mathbf{q}^T\mathbf{r}_t)]$. Differentiating with respect to $\mathbf{q}$ we have
	$\frac{\partial H(\mathbf{q})}{\partial \mathbf{q}}=\E\left[\mathbf{r}_tf(\mathbf{q}^T\mathbf{r}_t)\right]$ and $\frac{\partial^2 H(\mathbf{q})}{\partial \mathbf{q}^2}=\E\left[\mathbf{r}_t\mathbf{r}_t^Tf'(\mathbf{q}^T\mathbf{s}_t)\right]$.
	Hence by Taylor expansion around $\eta\mathbf{e}_i+\boldsymbol{\epsilon}$ given that $c_T \le \lVert \boldsymbol{\epsilon}\rVert^2$ we have
	\begin{align*}
		& H\left(\eta\mathbf{e}_i+\boldsymbol{\epsilon}\right) \\
		& = H(\eta\mathbf{e}_i)+\E\left[\mathbf{r}_t^Tf(\eta r_{it})\right]\boldsymbol{\epsilon}+\frac{1}{2}\boldsymbol{\epsilon}^T\E\left[\mathbf{r}_t\mathbf{r}_t^Tf'(\eta r_{it})\right]\boldsymbol{\epsilon}+o(\lVert \boldsymbol{\epsilon}\rVert^2) \\
		& = H(\eta\mathbf{e}_i)+ \E\left[r_{jt}f(\eta r_{it})\right]\epsilon_j + \sum_{j\neq i}\E\left[r_{jt}f(\eta r_{it})\right]\epsilon_j +\frac{1}{2}\E\left[f'(\eta r_{it})r_{it}^2\right]\epsilon_i^2 \\
		& +\frac{1}{2}\sum_{j\neq i}\sum_{l \neq i,j}\E[f'(\eta r_{it})r_{jt}r_{lt}]\epsilon_j\epsilon_l+\frac{1}{2}\sum_{j \neq i}\E\left[f'(\eta r_{it})r_{jt}^2\right]\epsilon_j^2+o(\lVert \boldsymbol{\epsilon}\rVert^2) 
	\end{align*}
	By the two conditions we thus have
	\begin{align*}
		& H\left(\eta\mathbf{e}_i+\boldsymbol{\epsilon}\right) \\
		& \le H(\eta\mathbf{e}_i)+ \E\left[r_{jt}f(\eta r_{it})\right]\epsilon_j+\sum_{j\neq i}\E\left[f(\eta r_{it})\right]c_T\epsilon_j +\frac{1}{2}\E\left[f'(\eta r_{it})r_{it}^2\right]\epsilon_i^2 \\
		& +\frac{1}{2}\sum_{j\neq i}\sum_{l \neq i,j}\E[f'(\eta r_{it})]c_T^2\epsilon_j\epsilon_l+\frac{1}{2}\sum_{j \neq i}\E\left[f'(\eta r_{it})r_{jt}^2\right]\epsilon_j^2+o(\lVert \boldsymbol{\epsilon}\rVert^2) 
	\end{align*}
	By the condition that $\mathbf{q}^T\mathbf{q}=\eta$, we can write $\epsilon_i=\sqrt{\left(\eta-\sum_{j\neq i}\epsilon_j^2\right)}-\eta$. 
	
	Since $(\eta-\gamma)^{1/2}=\eta-\gamma/2+o(\gamma^2)$, $\epsilon_i=-\frac{1}{2}\sum_{j \neq i}\epsilon_j^2$, which also implies that $\epsilon_i^2$ is of $o(\lVert \boldsymbol{\epsilon}\rVert^2)$. Finally, we conclude that
	\begin{align*}
		& H\left(\eta\mathbf{e}_i+\boldsymbol{\epsilon}\right) \\
		& \le H(\eta\mathbf{e}_i)+\frac{1}{2}\left(\sum_{j \neq i}\E\left[f'(\eta r_{it})r_{jt}^2\right]-\E\left[f(\eta r_{it})r_{it}\right]\right)\sum_{j \neq i}\epsilon_j^2+o(\lVert \boldsymbol{\epsilon}\rVert^2)
	\end{align*}
	which proves that $\eta \mathbf{e}_i$ is indeed the optimum point of $H$, i.e. $\E[F(\mathbf{w}^T\mathbf{s}_t)]$.
\end{proof}

Notice that the sign of the expectation of the random variable illustrated in Theorem \ref{stabilitythm1} implies how the objective function $F$ characterises the distribution of the random variable and we wish to maximise $\E[F(\mathbf{w}^T\mathbf{z})]$ in one category and minimise in another. Hence, we would have the following

\begin{thm}
	Let $\mathbf{z}=W\mathbf{s}=WA\mathbf{r}$ be a whitened data and $F$ is a smooth even function. Then the asymptotically stable points of the algorithm 1 in Section \ref{subsec:maxnong} include the i-th row of the inverse of the whitened mixing matrix $WA$ such that the corresponding independent components $r_i$ satisfy
	$$\E[r_if(r_i)-f'(r_i)]\left\{\E[F(r_i)]-\E[F(v)]\right\}>0$$
	where $v$ is a standardised gaussian variable.
\end{thm}

The maximization of nongaussianity approach naturally handles higher-dimensional cointegration systems by separating all independent components simultaneously. Each recovered component is tested for stationarity using the Augmented Dickey-Fuller test, with the stationary components corresponding to cointegrating relationships. The cointegration vectors are directly obtained from the rows of the demixing matrix associated with stationary components.

However, notice that the application of the ADF test to components recovered via an optimization procedure may lead to non-standard asymptotic distributions of the test statistic, a problem analogous to that encountered when testing OLS regression residuals for cointegration \cite{engle1987co}. The critical values used in this paper should therefore be interpreted with caution. A more rigorous analysis would require the simulation of specific critical values for this BSS-based procedure, which we leave for future research.

\begin{re}\label{rem:gaussian_limit}
	The optimization in (3.7) relies critically on the \textit{non-Gaussianity} of the stationary components. If all stationary components follow Gaussian distributions, the negentropy measure \(J(\mathbf{y})\) becomes zero, and the objective function \(\mathscr{G}(\cdot)\) fails to identify meaningful independent components. This limitation arises because Gaussian distributions are rotationally invariant under orthogonal transformations, making separation impossible. 
	
	However, this constraint does not invalidate our method in practical applications. Firstly, financial/economic time series frequently exhibit non-Gaussian features (e.g., heavy tails from volatility clustering). Secondly, cointegration residuals often demonstrate non-normality due to structural breaks or heteroskedasticity. Lastly, the decorrelation method (Section 3.2) remains valid for Gaussian systems when second-order statistics suffice. In scenarios with suspected Gaussianity, we recommend using Johansen's test or verifying non-normality via Jarque-Bera tests prior to applying our method.
\end{re}

\section{Numerical examples}\label{sec:appl}

In this section, we apply our global optimisation on several models and data sets and evaluate its effectiveness and robustness by comparing with some classical models such as the Johansen cointegration test. Note that the primary goal here is not to beat those well-developed algorithms, but to provide a new view point for considerations when dealing with different time series problems. 

\subsection{Test for unit-root nonstationary simulated series}\label{sec:simulatedseries1}

\hfill

In this section, we carry out a Monte Carlo experiment with 1000 repetitions on a simulated time series as a demonstration of our methods. We will also use bias and mean square error (MSE) for comparison in different sample sizes.

We shall begin with a simulated series, consider the following VARMA(1,1) model:
\begin{align}
	\begin{pmatrix}
		s_{1t}  \\
		s_{2t}
	\end{pmatrix}-
	\begin{pmatrix}
		0.5 & -1.0  \\
		-0.25 & 0.5
	\end{pmatrix}
	\begin{pmatrix}
		s_{1,t-1}  \\
		s_{2,t-1}
	\end{pmatrix}=
	\begin{pmatrix}
		\varepsilon_{1t}  \\
		\varepsilon_{2t}
	\end{pmatrix}-
	\begin{pmatrix}
		0.2 & -0.4  \\
		-0.1 & 0.2
	\end{pmatrix}
	\begin{pmatrix}
		\varepsilon_{1,t-1}  \\
		\varepsilon_{2,t-1}
	\end{pmatrix}, \label{arma_data}
\end{align}
where $(\boldsymbol{\varepsilon}_t)$ is a multivariate white noise series.

The above model is unit-root nonstationary \cite{tsay2005analysis} and a time plots of the two components are shown in Figure~\ref{ex1:original}. However, it can be cointegrated in the sense that once we apply the following linear transformation upon the original series:
\begin{align}
	\begin{pmatrix}
		z_{1t}  \\
		z_{2t}
	\end{pmatrix} \equiv
	\begin{pmatrix}
		1  & -2  \\
		0.5 & 1
	\end{pmatrix}
	\begin{pmatrix}
		s_{1t}  \\
		s_{2t}
	\end{pmatrix} := L
	\begin{pmatrix}
		s_{1t}  \\
		s_{2t}
	\end{pmatrix}, \nonumber
\end{align}
\begin{align}
	\begin{pmatrix}
		\delta_{1t}  \\
		\delta_{2t}
	\end{pmatrix} \equiv
	\begin{pmatrix}
		1  & -2  \\
		0.5 & 1
	\end{pmatrix}
	\begin{pmatrix}
		\varepsilon_{1t}  \\
		\varepsilon_{2t}
	\end{pmatrix} := L
	\begin{pmatrix}
		\varepsilon_{1t}  \\
		\varepsilon_{2t}
	\end{pmatrix}, \nonumber
\end{align}
which gives
\begin{align}
	\begin{pmatrix}
		z_{1t}  \\
		z_{2t}
	\end{pmatrix}-
	\begin{pmatrix}
		1.0 & 0  \\
		0 & 0
	\end{pmatrix}
	\begin{pmatrix}
		z_{1,t-1}  \\
		z_{2,t-1}
	\end{pmatrix}=
	\begin{pmatrix}
		\delta_{1t}  \\
		\delta_{2t}
	\end{pmatrix}-
	\begin{pmatrix}
		0.4 & 0  \\
		0 & 0
	\end{pmatrix}
	\begin{pmatrix}
		\delta_{1,t-1}  \\
		\delta_{2,t-1}
	\end{pmatrix}.  \label{arma_transform_data}
\end{align}
where we can see that second component $z_{2t}$ of the transformed model is a white-noise series. We then have the cointegration vector $\boldsymbol{\beta} = (0.5,1.0)^T$ as $z_{2t}=0.5s_{1t} + 1.0s_{2t}$ is a stationary series.

Regarding the above model, we apply a blind source separation using decorrelation method on the two time series components to separate out a nonstationary time series $r_t$ and a stationary series $\varepsilon_t$ as shown in Fig.~\ref{ex1:original}.

Now, if $z_t$ denotes the cointegrated series from $s_{1t}$ and $s_{2t}$, then we can estimate the linear regression $z_t = \beta_1s_{1t} + \beta_2s_{2t}$. To obtain an estimate for $\boldsymbol{\beta} = (\beta_1,\beta_2)$, we use the least square method according to the main algorithm for one signal and one noise (Section \ref{sec:opt2*2}). We then have $\beta_1 = 0.5128$ and $\beta_2 = 1.0351$ where the cointegration vector $\boldsymbol{\beta} = (0.5128,1.0351)$. Note that in this case, a cointegration vector (if exists in bivariate settings) is unique up to scaling. Hence, we may normalise the cointegration vector to $\hat{\boldsymbol{\beta}}=(0.4954, 1.0)^T$. One may see that the relative error in our estimation with the original cointegration vector is around $0.9\%$. Note that the above example is an application of the global optimisation algorithm on a simulated series. In real situations, which will be shown in later sections, we shall apply the algorithm directly on a set of data and obtain an estimation of the cointegration vector.

In real life trading, it is generally favourble to sell overvalued securities and purchase undervalued securities. However, prices behave in a random-walk manner, which makes it difficult to be predicted, while a linear combination of series produces a stationary series, which means that the price is mean-reverting. Furthermore, mean value obtained from cointegration gives trading opportunities.

In the examples illustrated above, if we view $s_{1t}$ and $s_{2t}$ as Stock 1 and Stock 2 respectively, we may see from Fig.~\ref{ex1:original} that the prices fluctuate quite randomly and could be very hard to predict. However, the stationarity of the linear combination (or cointegrated series) $z_t = \beta_1s_{1t} + \beta_2s_{2t}$ implies that $z_t$ is mean-reverting. Now, from the portfolio $Z$ by buying $\beta_1$ share of Stock 1 and buying $\beta_2$ shares of Stock 2, the return of the portfolio $Z$ for a given period $\Delta t$ is $r(\Delta t) = z_{t + \Delta t} - z_t$, which is the increment of the stationary series $\{z_t\}$
from $t$ to $t + \Delta t$. We have obtained a direct link of the portfolio to a stationary time series whose forecasts we can predict. This also highlights the importance of cointegration vector in convergence trading.

\begin{re}\label{mean_reverting}
	The speed of mean-reverting of $z_{t}$ plays an important role in practical treading. One should consider the profit and the transaction costs.
\end{re}

\begin{figure*}[t!]
	\centering
	\begin{subfigure}[t]{0.5\textwidth}
		\centering
		\includegraphics[width=1\textwidth]{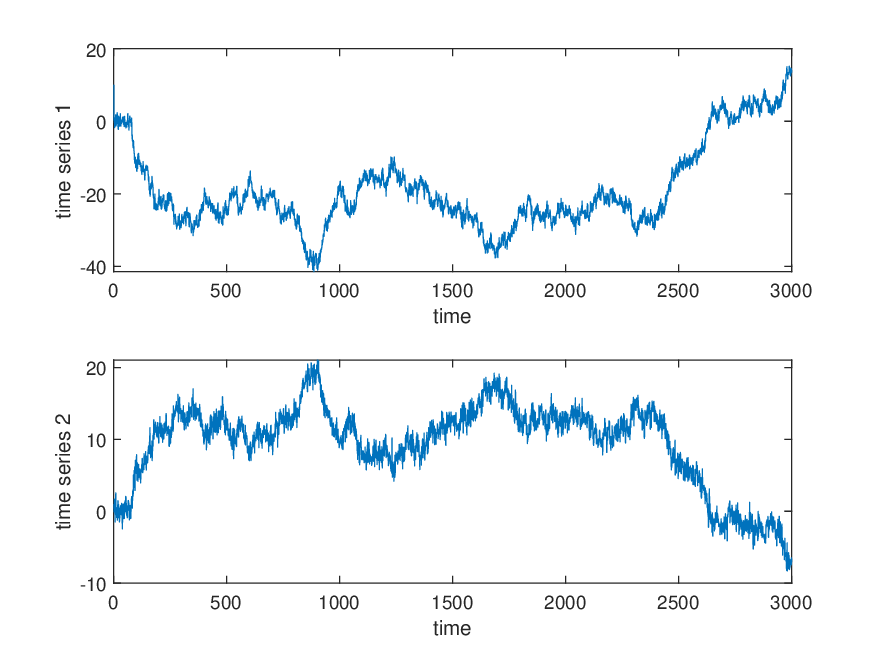}
		\caption{Time plots of the VARMA(1,1) model}
		\label{ex1:original}
	\end{subfigure}%
	~ 
	\begin{subfigure}[t]{0.5\textwidth}
		\centering
		\includegraphics[width=1\textwidth]{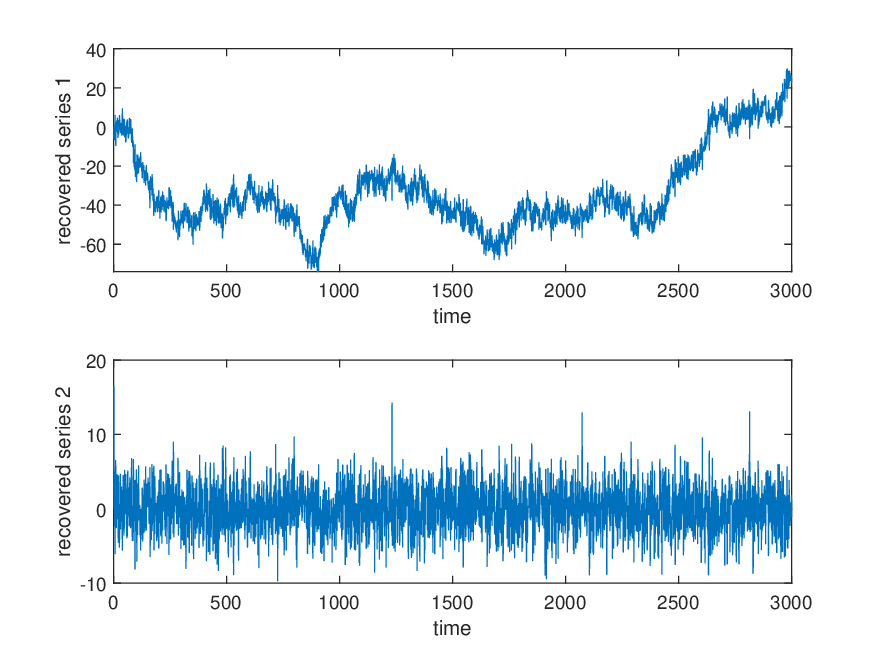}
		\caption{Time plots of recovered series}
		\label{ex1:recovered}
	\end{subfigure}
	\caption{Cointegration test by decorrelation}
\end{figure*}

We may also evaluate the cointegration vector through maximisation of nongaussianity as introduced in section \ref{subsec:maxnong} and \ref{sec:optn*n}. In view of the construction of the ARMA model \eqref{arma_data}. The optimal solutions to the optimisation problem \eqref{nongaussianopt} $\hat{W}$ are given in Table~1, where we use the smooth even functions \eqref{approxf1} and \eqref{approxf2} respectively. Fig.~\ref{fig:ex2_arma_data} - Fig.~\ref{fig:ex2_combined_data} (Notice that the time plots of the same VARMA model series could be completely different as random generative processes are included in the construction of the white noise series in the moving-average model) shows the outcome of maximisation of nongaussianity followed by the algorithm 2 in Section~\ref{sec:optn*n} using the objective function $F(x)=\frac{1}{a}\ln\cosh{(ax)}$ and taking $a=1$.
\begin{table}\label{maxnongausssimulatedseries}\caption{Test results by different objective functions}
	\begin{center}
		\begin{tabular}{|c|c|c|}
			\hline  &$F(x)=\ln\cosh(x)$&$F(x)=-e^{-\frac{x^2}{2}}$ \\
			\hline  $\hat{W}$ & $\begin{pmatrix} 0.8946 & -0.4469  \\ 0.4514 & 0.8923 \end{pmatrix}$& $\begin{pmatrix} 0.8948 & -0.4464  \\ 0.4464 & 0.8948 \end{pmatrix}$\\
			\hline
		\end{tabular}
	\end{center}
\end{table}

Notice that $\mathbf{z}_t=\mathbf{w}^T\mathbf{s}_t=\mathbf{q}^T\mathbf{r}_t$ is exactly a linear combination of the independent components of $\mathbf{r}_t$, hence it has the same number of stationary (if any) and nonstationary components of $\mathbf{r}_t$. The objective functions respectively give the desired cointegration vector $\boldsymbol{\beta}^{(1)} = (0.4514, 0.8923)^T$ and $\boldsymbol{\beta}^{(2)} = (0.4464, 0.8948)^T$ as the second component of the combined series is stationary. Notice that a cointegration vector is a linear combination of components of a time series that gives a stationary process, thus any scalar multiple of the original cointegration vector still produces a stationary linear combination, which is also a cointegration vector. Comparing our results with the analytical solution $\boldsymbol{\beta}^{(1)} = (1/2, 1)^T$, after scaling the second term to $1$, we have the following cointegration vectors: $\boldsymbol{\beta}^{(1)} = (0.5062, 1)^T$ and $\boldsymbol{\beta}^{(2)} = (0.4989, 1)^T$. The relative error is around 1.24\% and 0.22\%. The accuracy of our algorithm may rely on the choice of the objective function, which could be highly empirical. Comparing relative errors of maximising nongaussianity and that of decorrelation, one may realise difference in performances from different choices of the objective functions and one shall also consider other factors such as the random generative processes of the white noise series in the construction of the simulated series. Different performances and slight variations in the robustness of both algorithms may also be expected in different scenarios when we tackle tests in other time series. However, it can be seen that both methods provide promising and reliable results for lower-dimensional cases. 

\begin{figure*}[t!]
	\centering
	\begin{subfigure}[t]{0.5\textwidth}
		\centering
		\includegraphics[width=1\textwidth]{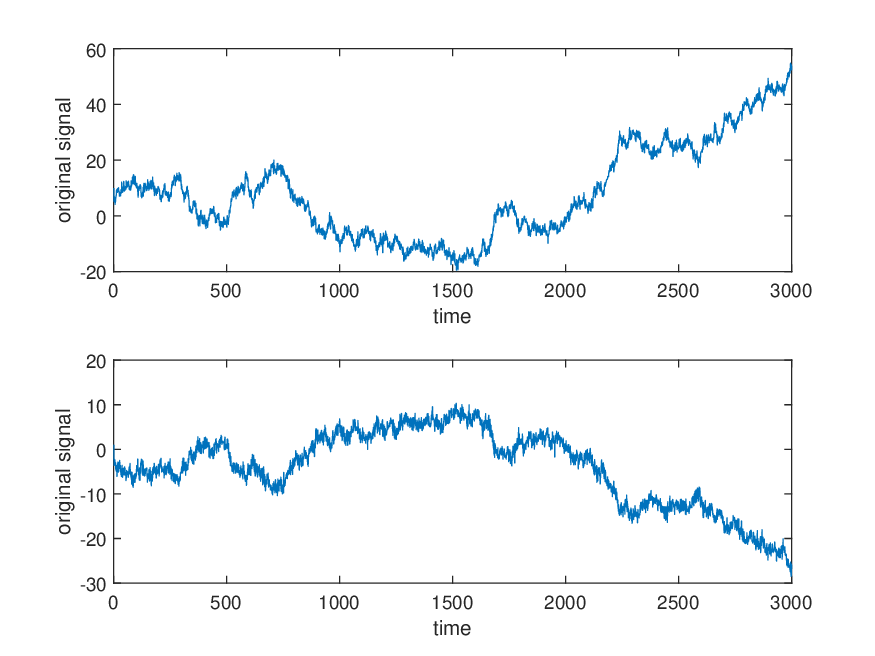}
		\caption{Time plots of the VARMA(1,1) model}
		\label{fig:ex2_arma_data}
	\end{subfigure}%
	~ 
	\begin{subfigure}[t]{0.5\textwidth}
		\centering
		\includegraphics[width=1\textwidth]{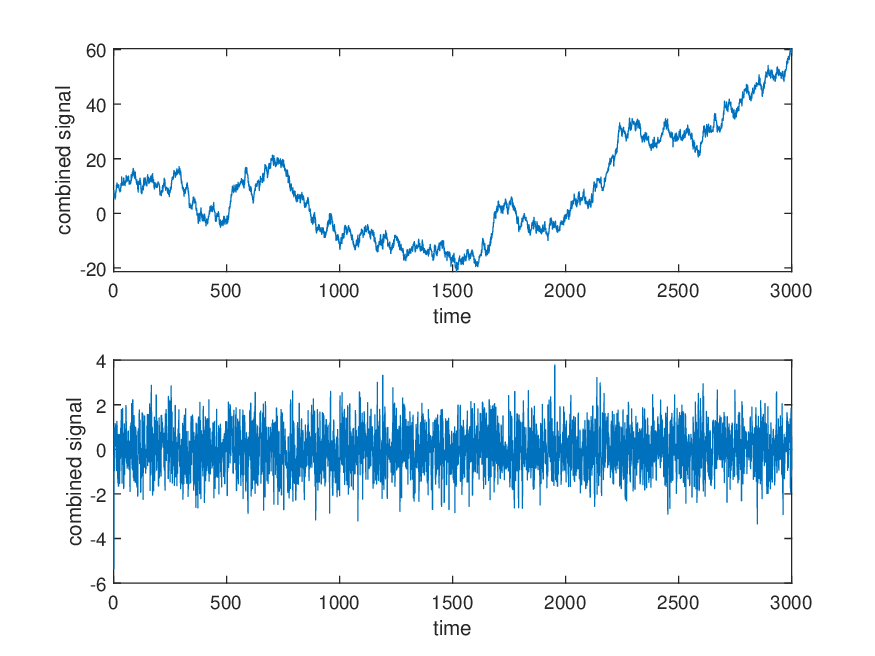}
		\caption{Time plots of combined series}
		\label{fig:ex2_combined_data}
	\end{subfigure}
	\caption{Cointegration test by maximising nongaussianity}
\end{figure*}

\subsection{Monte Carlo Simulation Results}\label{sec:monte_carlo}

\hfill

After implementing our methods on the simulated unit-root nonstationary series, we may also compare our results to existing tests such a the Johansen's Cointegration test. As shown in the Table~2. for test results by using Johansen's trace test applied on the VARMA model introduced in section \ref{sec:simulatedseries1} with 3000 observations.
\begin{table}\label{johansentable}\caption{Test results by Johansen's Cointegration Test}
	\begin{center}
		\begin{tabular}{|c|c|c|c|c|c|}
			\hline  $r$&$h$&stat&c-value&p-value&eigenvalue \\
			\hline  $0$&$1$&$2390.2790$&$15.4948$&$0.0010$&$0.5476$\\
			\hline  $1$&$0$&$11.4968$&$3.8415$&$0.0010$&$0.0038$\\
			\hline
		\end{tabular}
	\end{center}
\end{table}

The test shows cointegration of rank 1 indicating the unique cointegration vector nature in bivariate cases, and from the cointegration equations, the test exhibits the cointegration vector $(0.4445,0.8892)^T$. Normalising the cointegration vector, we have $\boldsymbol{\beta}^{(1)} = (0.4991,1)^T$.

A Monte Carlo simulation is then carried out to compare the performance of our methods and Johansen's in different sample sizes. As shown in Table~3.

\begin{table}[htbp]
	\centering
	\caption{Monte Carlo Simulation Results: Bias and MSE for Cointegration Vector Estimation}
	\label{tab:monte_carlo_results}
	\begin{tabular}{c l S[table-format=1.4] S[table-format=1.2e-1]}
		\toprule
		\textbf{Sample Size (T)} & \textbf{Method} & \multicolumn{1}{c}{\textbf{Bias($\beta_1$)}} & \multicolumn{1}{c}{\textbf{MSE($\beta_1$)}} \\
		\midrule
		10 & Decorrelation & 6.3220 & 4.64e+03 \\
		& Max. Nongaussianity & 1.4708 & 4.05e+01 \\
		& Johansen & 2.2996 & 1.00e+02 \\
		\midrule
		15 & Decorrelation & 4.0491 & 1.85e+03 \\
		& Max. Nongaussianity & 0.9306 & 3.00e+00 \\
		& Johansen & 1.3117 & 4.84e+01 \\
		\midrule
		20 & Decorrelation & 2.0884 & 5.79e+01 \\
		& Max. Nongaussianity & 1.0794 & 1.04e+02 \\
		& Johansen & 1.2554 & 2.48e+02 \\
		\midrule
		25 & Decorrelation & 3.6321 & 2.98e+03 \\
		& Max. Nongaussianity & 0.7921 & 3.62e+01 \\
		& Johansen & 0.3454 & 2.08e+00 \\
		\midrule
		30 & Decorrelation & 1.5528 & 4.47e+01 \\
		& Max. Nongaussianity & 0.5343 & 1.37e+00 \\
		& Johansen & 0.2072 & 3.66e-01 \\
		\midrule
		40 & Decorrelation & 0.7040 & 2.61e+01 \\
		& Max. Nongaussianity & 0.3777 & 2.66e-01 \\
		& Johansen & 0.1498 & 9.60e-01 \\
		\midrule
		50 & Decorrelation & 0.5428 & 3.05e+01 \\
		& Max. Nongaussianity & 0.3048 & 1.82e-01 \\
		& Johansen & 0.0906 & 4.32e-02 \\
		\midrule
		100 & Decorrelation & 0.1355 & 7.10e-01 \\
		& Max. Nongaussianity & 0.1809 & 8.47e-02 \\
		& Johansen & 0.0367 & 2.62e-03 \\
		\midrule
		3000 & Decorrelation & 0.0022 & 1.34e-05 \\
		& Max. Nongaussianity & 0.0089 & 3.85e-04 \\
		& Johansen & 0.0011 & 2.18e-06 \\
		\bottomrule
	\end{tabular}
\end{table}

\begin{figure}
	\centering
	\includegraphics[totalheight=8cm]{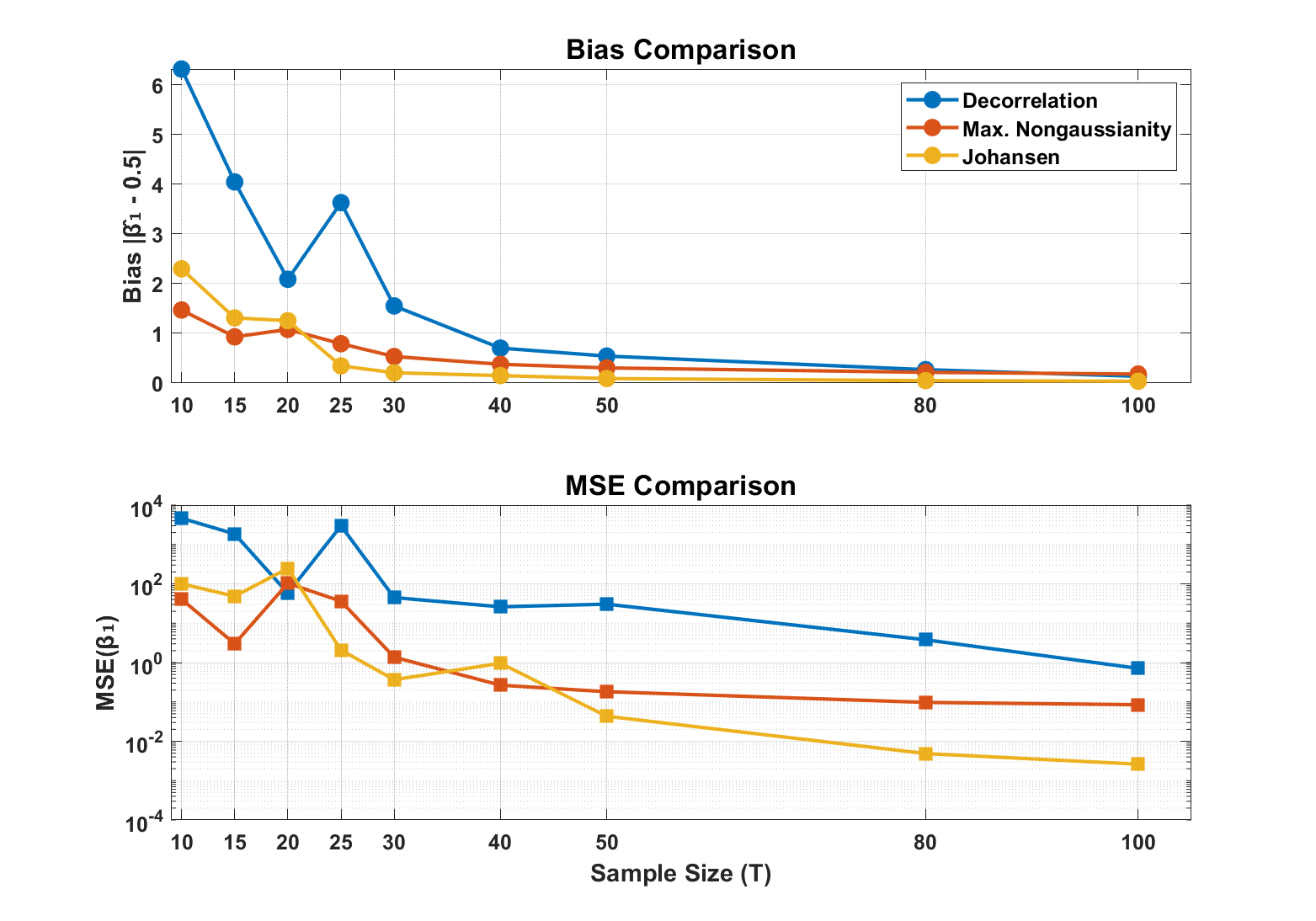}
	\caption{Monte Carlo Simulation Error Comparison}
	\label{fig:montecarloerror}
\end{figure}

In the comparative analysis conducted, it is evident that Johansen's method continues to exhibit superior precision and accuracy, particularly when applied to datasets characterized by large sample sizes. This observation is further supported by the comprehensive data presented in Table~3, which meticulously details the performance outcomes of the three aforementioned methods across varying sample sizes. However, a nuanced examination reveals that in instances where the sample size is constrained and resources are limited, the Decorrelation and Maximisation of Nongaussianity techniques demonstrate a marginally enhanced performance profile \cite{chan2013advantages}.

\begin{table}\label{testcomparisons}\caption{Test results by different sample sizes}
	\begin{center}
		\begin{tabular}{|c|c|c|c|}
			\hline
			$T$&{Johansen}&{Decorrelation}&{Max. nongauss} \\
			\hline   $50$&$(0.4247,1)^T$&$(0.4861,1)^T$&$(0.4678,1)^T$\\
			\hline
			$3000$&$(0.4991,1)^T$&$(0.4954,1)^T$&$(0.5008,1)^T$\\
			\hline
		\end{tabular}
	\end{center}
\end{table}

The significance and motive for introducing maximising nongaussianity is that it is possible for us to tackle higher dimensional cases. The following example shall illustrate the performance of our algorithm \ref{mainalgorithm} in a higher dimension.

\subsection{Test for higher dimensional finite samples}

\hfill

The robustness of maximising nongaussianity can be seen from its preponderance in dealing with time series or signals with high dimensions. Consider the following time series $\mathbf{s}_t$ defined by
\begin{equation}\label{ex2}
	\begin{pmatrix} s_{1t} \\ s_{2t} \\ s_{3t} \end{pmatrix} - \begin{pmatrix} -0.3 & -0.03 & -0.375 \\ -1.88 & -1.472 & -2.08 \\ 1.07 & 0.428 & 1.27 \end{pmatrix}\begin{pmatrix} s_{1,t-1} \\ s_{2,t-1} \\ s_{3,t-1} \end{pmatrix}=\begin{pmatrix} \varepsilon_{1t} \\ \varepsilon_{2t} \\ \varepsilon_{3t} \end{pmatrix} - \begin{pmatrix} 0.75 & 0.075 & 0.9375 \\ 4.7 & 3.68 & 5.2 \\ -2.675 & -1.07 & -3.175 \end{pmatrix}\begin{pmatrix} \varepsilon_{1,t-1} \\ \varepsilon_{2,t-1} \\ \varepsilon_{3,t-1} \end{pmatrix}
\end{equation}
By applying the transformation $L=\begin{pmatrix} 1 & -2 & -0.5 \\ 0.5 & 1 & 4 \\ 2 & 0.25 & 1\end{pmatrix}$ to the left-hand-side of the VARMA model, we have
\begin{align}
	\begin{pmatrix} z_{1t} \\ z_{2t} \\ z_{3t} \end{pmatrix} - \begin{pmatrix} 13 & 12 & 14 \\ 10 & 1 & 12.5 \\ 0 & 0 & 0 \end{pmatrix}\begin{pmatrix} z_{1,t-1} \\ z_{2,t-1} \\ z_{3,t-1} \end{pmatrix}=\begin{pmatrix} \delta_{1t} \\ \delta_{2t} \\ \delta_{3t} \end{pmatrix} - \begin{pmatrix} -7.3125 & -6.75 & -7.875 \\ -5.625 & -0.5625 & -7.03125 \\ 0 & 0 & 0 \end{pmatrix}\begin{pmatrix} \delta_{1,t-1} \\ \delta_{2,t-1} \\ \delta_{3,t-1} \end{pmatrix}
\end{align}
where $\mathbf{z}_t=L\mathbf{s}_t$ and $\boldsymbol{\delta}_t=L\boldsymbol{\varepsilon}_t$ for any $t \in \R$. The third component, $z_{3t}$, of the combined series $\mathbf{z}_t$ is a stationary series as $\boldsymbol{\delta}_t=L\boldsymbol{\varepsilon}_t$ is a stationary series. Hence $\boldsymbol{\beta} = (2, 0.25, 1)^T$ is the desired analytical cointegration vector. The following analysis shows the implementation and accuracy of our method.

Applying the optimisation (maximising nongaussianity) method to the VARMA model \eqref{ex2} using the objective function $F(x)=\ln \cosh(x)$ and after diagonalising the coefficient matrix, we obtain the matrix solution to the optimisation problem \eqref{nongaussianopt}:
$$\hat{W}=\begin{pmatrix} -0.0402 & 0.9856 & -0.1640 \\ 0.4652 & -0.1268 & -0.8761 \\ 0.8843 & 0.1115 & 0.4534 \end{pmatrix}$$
One may also check that the rows of the coefficient matrix $\hat{W}$ are orthogonal to each other. Normalising the last row of $\hat{W}$ which corresponds to the last component of the combined series $\boldsymbol{\varepsilon}_t$ that produces a stationary process, we have the estimated cointegration vector $\boldsymbol{\hat{\beta}} = (1.9503, 0.2459, 1)^T$ compared to the analytical solution $\boldsymbol{\beta} = (2, 0.25, 1)^T$, the relative error is around $2.06\%$.

\subsection{High-Dimensional Example with Multiple Cointegration Vectors} \label{sec:high_dim}

\hfill

To demonstrate the capability of our method in higher-dimensional settings with multiple cointegration relationships, we consider a 4-dimensional system with two cointegration vectors. The data generating process is defined as follows:

Let \(\mathbf{c}_t = (c_{1t}, c_{2t}, c_{3t}, c_{4t})^\top\) where:
\begin{align*}
	c_{1t} &= c_{1,t-1} + \varepsilon_{1t}, \quad \varepsilon_{1t} \sim \mathcal{N}(0, 0.01) \\
	c_{2t} &= c_{2,t-1} + \varepsilon_{2t}, \quad \varepsilon_{2t} \sim \mathcal{N}(0, 0.01) \\
	c_{3t} &\sim  t(5) \\
	c_{4t} &\sim  t(5)
\end{align*}
The components \(\mathbf{c}_t\) are linearly mixed through a random invertible matrix \(A \in \mathbb{R}^{4\times4}\) to produce the observed series:
\[
\mathbf{s}_t = A \mathbf{c}_t
\]
We generate \(T=3000\) observations after a burn-in period of 500 points. The true cointegration vectors correspond to the last two rows of \(A^{-1}\), which represent the stationary non-Gaussian components (\(c_{3t}, c_{4t}\)).

Applying the maximization of nongaussianity method (Algorithm 2, Section 3.3) with \(F(x) = \ln\cosh(x)\), we obtain the demixing matrix \(\hat{W}\). The recovered components are tested for stationarity using the Augmented Dickey-Fuller test at 1\% significance level. The two components with the most negative test statistics (strongest evidence of stationarity) are identified as cointegration relationships. 

Again, standard critical values may lead to size distortion due to applying the ADF on time series resulting from the optimisation procedure. Hence, a more negative value should be used for residuals than usual. In our particular case, the critical value for both cointegration vectors are -11.3892 and -10.3729 respectively, which are widely more negative than usual critical values, while a rigorous analysis on critical values used for BSS procedures should be implemented via a wider range of simulation.

Results from a representative simulation show:
\begin{align*}
	&\text{\textbf{Estimated cointegration vectors:}} \\
	&\hat{\boldsymbol{\beta}}_1 = [1.0000, 0.6356, 0.4620, -0.2958]^\top \quad  \\
	&\hat{\boldsymbol{\beta}}_2 = [1.0000, 0.2799, 0.8265, 0.1067]^\top \quad  \\
	&\text{\textbf{True cointegration vectors:}} \\
	&\boldsymbol{\beta}_1 = [1.0000, 0.6339, 0.4728, -0.2852]^\top \\
	&\boldsymbol{\beta}_2 = [1.0000, 0.3021, 0.8325, 0.1062]^\top
\end{align*}
The relative errors in the coefficients are below \(0.05\%\), confirming the method accurately recovers multiple cointegration vectors in high dimensions. Figure 6 visualizes the components and their autocorrelation functions.
\begin{figure}
	\centering
	\includegraphics[width=\textwidth]{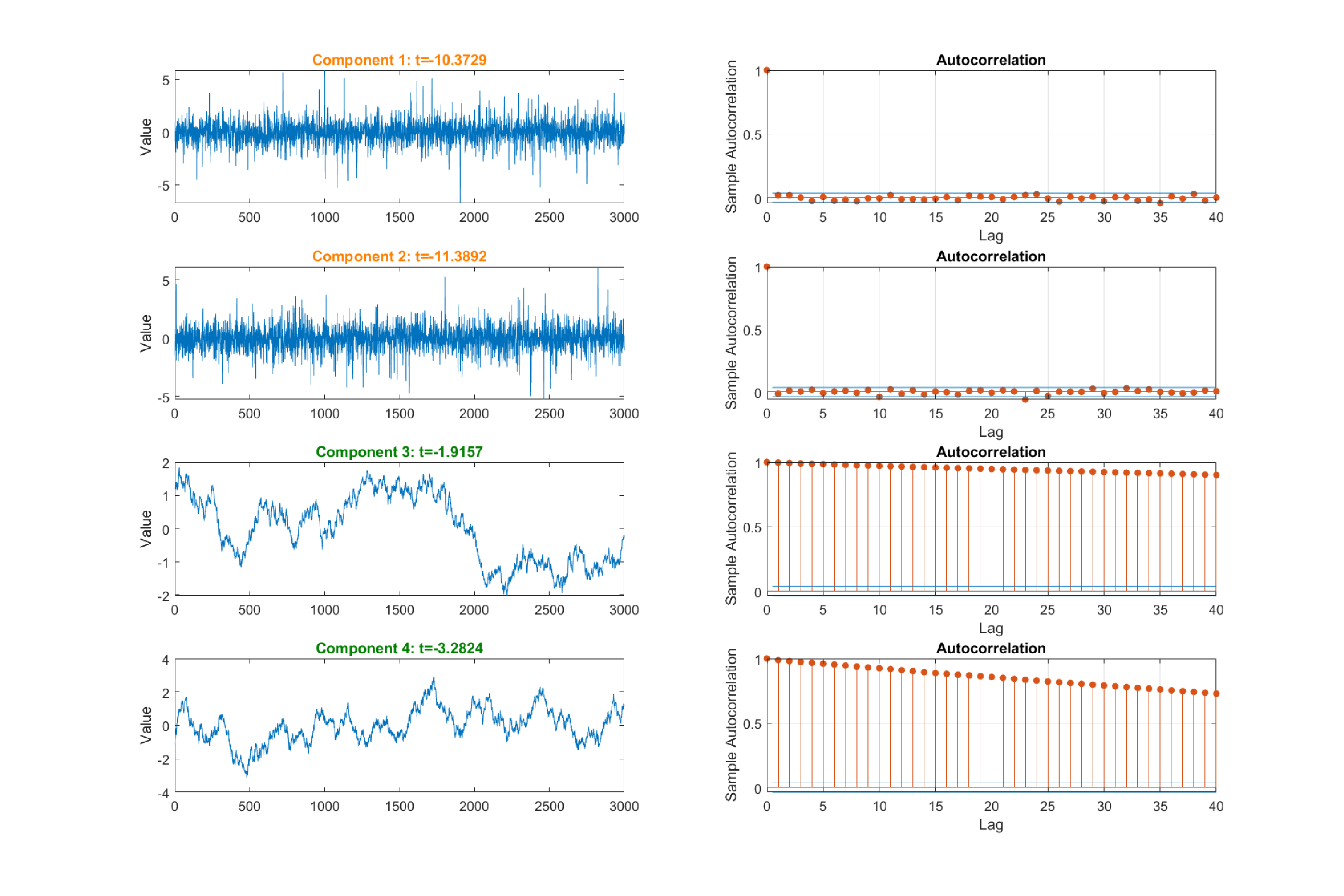}
	\caption{Recovered components via maximization of nongaussianity. Components 3 and 4 (bottom) show stationarity and match the true non-Gaussian sources.}
	\label{fig:high_dim}
\end{figure}

\subsection{Empirical data demonstration: oil prices}\label{sec:simulatedseries}

\hfill

In this section, our methodology is applied to authentic oil price data, enabling a comprehensive evaluation of its robustness and precision in analyzing long-term stochastic trends when compared to alternative methodologies. The dataset utilized encompasses the monthly oil prices of both Brent and Dubai spanning the years from 1960 to 2022. The intricacies of oil price dynamics have long been a focal point of research endeavors, primarily due to the multifaceted influences that shape its trajectory, including geopolitical events and the availability of alternative resources. Noteworthy studies by Pesaran and Timmermann (2005) \cite{pesaran2005small}, Narayan (2007) \cite{narayan2007modelling}, and Aloui and Mabroul (2010) \cite{aloui2010value} have delved into the potential for cointegration within oil price datasets, shedding light on the evolving nature of this critical commodity. Central to this investigation is the underlying premise that the progressive scarcity of a finite resource like oil inherently predisposes its price to exhibit a sustained upward trend, thereby creating fertile ground for the emergence of cointegration relationships over extended temporal horizons.

Fig.~\ref{fig:ex3_oil_prices} shows time plots of Brent and Dubai oil prices respectively, with 756 monthly data points. First, we shall see the application of decorrelation on cointegration of oil prices. The undiscovered components are shown in Fig.~\ref{fig:ex3_recovered_data} with the second series being a stationary process. After some algebraic manipulations we obtain the following mixing matrix $\begin{pmatrix} -1.9162 & 5.0396 \\ 3.1407 & -3.2599 \end{pmatrix}$ with the second row corresponding to the cointegration vector. After normalising, we have $\boldsymbol{\beta}=(1, -1.038)^T$.

Applying maximisation of nongaussianity following the algorithm \eqref{mainalgorithm}, we may also get combined signals which represent linear combinations of the original two series, as shown in Fig.~\ref{fig:ex4_combined_data}. From the corresponding row of the coefficient matrix, we obtain the cointegration vector $(-0.6905, 0.7234)^T$ and after normalising we have $\boldsymbol{\beta}=(1, -1.048)^T$.

Test for cointegration by Johansen's cointegration test is also implemented on the empirical data. The test suggests rank 1 cointegration in the series by trace test. From the cointegration equation result from Johansen's cointegration test, we have the cointegration vector $(0.7202, -0.7416)^T$ and after normalising we have $\boldsymbol{\beta}=(1, -1.030)^T$.

\begin{center}
	\begin{tabular}{|c|c|c|c|c|c|}
		\hline  $r$&$h$&stat&c-value&p-value&eigenvalue \\
		\hline  $0$&$1$&$83.2540$&$15.4948$&$0.0010$&$0.1018$\\
		\hline  $1$&$0$&$2.1571$&$3.8415$&$0.1425$&$0.0029$\\
		\hline
	\end{tabular}
\end{center}

\begin{figure*}[t!]
	\centering
	\begin{subfigure}[t]{0.5\textwidth}
		\centering
		\includegraphics[width=1\textwidth]{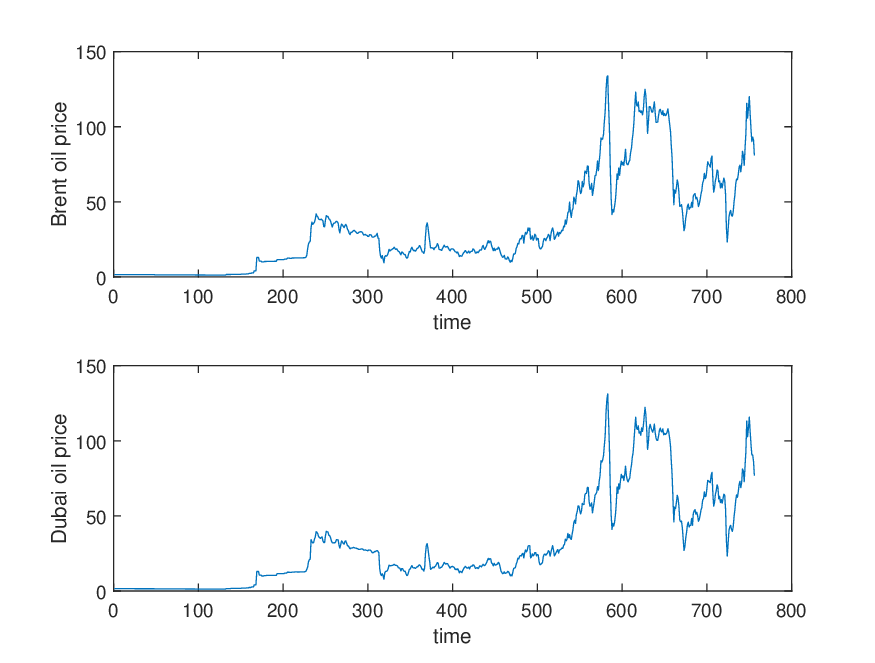}
		\caption{Time plots of oil prices}
		\label{fig:ex3_oil_prices}
	\end{subfigure}%
	~ 
	\begin{subfigure}[t]{0.5\textwidth}
		\centering
		\includegraphics[width=1\textwidth]{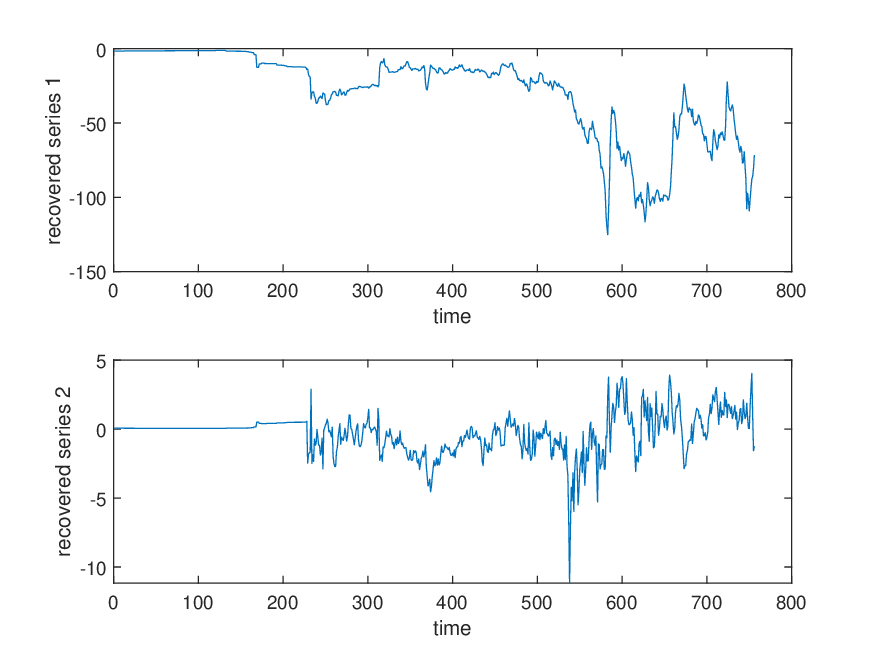}
		\caption{Recovered series after decorrelation}
		\label{fig:ex3_recovered_data}
	\end{subfigure}
	\caption{Cointegration test by decorrelation}
\end{figure*}	

\begin{figure*}[t!]
	\centering
	\begin{subfigure}[t]{0.5\textwidth}
		\centering
		\includegraphics[width=1\textwidth]{oil_prices}
		\caption{Time plots of oil prices}
		\label{fig:ex4_oil_prices}
	\end{subfigure}%
	~ 
	\begin{subfigure}[t]{0.5\textwidth}
		\centering
		\includegraphics[width=1\textwidth]{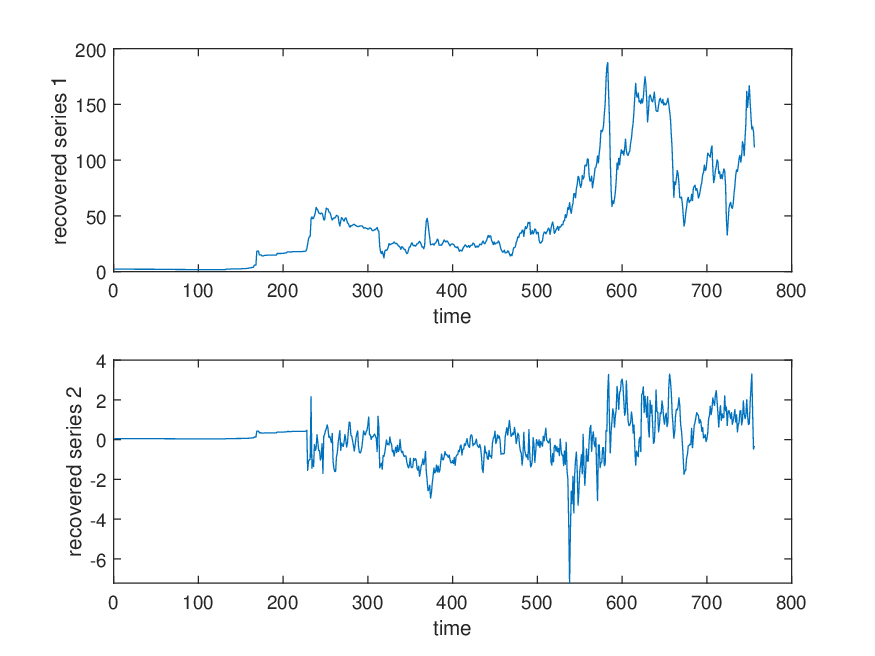}
		\caption{Recovered series}
		\label{fig:ex4_combined_data}
	\end{subfigure}
	\caption{Cointegration test by maximising nongaussianity}
\end{figure*}	

\section{Conclusion}
Two methods inspired by blind source separation are proposed in this research paper and are utilised in cointegration between time series in different dimensions. First and foremost, the decorrelation method is suitable for lower dimensional cases, typically involving 2 variables. The essence of the method is that it breaks down the problem into a polynomial form, and decorrelation simplifies the process and allows for easy solution. The optimisation program employed in this method demonstrates that once a stationary series is separated from a time series, tracing the cointegration vector becomes a straightforward task.

Furthermore, we have introduced how the maximisation of nongaussianity can be applied to test cointegration, which goes beyond bivariate cointegration. This method has shown its effectiveness in discovering the underlying stationary components of a mixed observed time series, thereby facilitating the subsequent process of finding cointegration. Generally, its preponderance clearly lies in its applicability to higher dimensional scenarios and overall satisfactory and promising accuracy.

Acknowledging the existence of well-developed cointegration tests, such as Johansen's cointegration test, note that the intention of this article is not to challenge the robustness of these tests, but rather to provide a new perspective on the role of independent components in time series analysis and the concept of cointegration. The advantages of utilizing independent components in various scenarios have been highlighted in terms of short-term statistics. We shall also emphasise the empirical advantage of these methods such that the order of the cointegration is not required as prior knowledge for cointegration and their convenience to be applied directly on the data sets provided for the time series compared to the Johansen's Cointegration Test. The research paper contributes to the field of time series analysis, numerical analysis and optimisation. By considering the advantages and limitations of these methods, valuable insights into the role of independent components in time series analysis and cointegration can be found.

\section{Acknowledgment}

The research of the author Zhiwen Zhang is partially supported by an R\&D Funding Scheme from the HKU-SCF FinTech Academy (HKU).

\section{Appendix}\label{sec:appendix}

\textbf{Johansen's Cointegration Test}

The most common practice of cointegration is using the Johansen's Cointegration Test \cite{johansen1991estimation}. Consider a VAR(k) model. Writing the model in Vector Error Correlation Model (VECM) form, we have 
\begin{equation}\label{VECM}
	\Delta\mathbf{s}_t=\boldsymbol{\mu}+\sum_{j=1}^{k-1}{\Gamma_j\Delta\mathbf{s}_{t-j}}+\Pi\mathbf{s}_{t-k}+\boldsymbol{\epsilon}_t
\end{equation}
Given that $\rank{\left(\Pi\right)}=r$, we can write $\Pi=\alpha\beta^T$ such that $\alpha$ and $\beta$ are matrices of dimension $r\times k$. Notice that the decomposition may not be unique. Denote $\mathbf{z}_{0t}=\Delta\mathbf{s}_t$, $\mathbf{z}_{1t}=\left(\Delta\mathbf{s}_{t-1}^T, \ldots, \Delta\mathbf{s}_{t-k+1}^T,\mathbf{1}^T\right)$ and $\mathbf{z}_{kt}=\Delta\mathbf{s}_{t-k}$. Then \eqref{VECM} can be transformed into
\begin{equation}\label{transformedVECM}
	\mathbf{z}_{0t}=\Gamma\mathbf{z}_{1t}+\alpha\beta^T\mathbf{z}_{kt}+\boldsymbol{\epsilon}_t
\end{equation}
where $\Gamma=\left(\Gamma_1, \ldots, \Gamma_{k-1},\boldsymbol{\mu}\right)$. Define the product moment matrix $M_{ij}=\frac{\sum_{t=1}^{T}{\mathbf{z}_{it}\mathbf{z}_{jt}^T}}{T}$ where $i, j = 0,1,k$. By regressing $\mathbf{z}_{it}$ on $\mathbf{z}_{1t}$, we have the residual $\mathbf{r}_{it}=\mathbf{z}_{it}-M_{i1}M_{11}^{-1}\mathbf{z}_{1t}$ and hence the residual sum of the squares from regressing $\mathbf{z}_{0t}$ and $\mathbf{z}_{kt}$ on $\mathbf{z}_{1t}$, $S_{ij}=\frac{\sum_{t=1}^{T}\mathbf{r}_{it}\mathbf{r}_{jt}^T}{T}$.

Johansen (1991) shows that based on the eigenvectors $\left(\mathbf{\hat v}_1, \ldots, \mathbf{\hat v}_r\right)$ derived from the equation
$$\det\left(\lambda S_{kk}-S_{k0}S_{00}^{-1}S_{0k}\right)=0$$
corresponding to the eigenvalues $\hat \lambda_1 > \ldots > \hat \lambda_n$.

We may conduct the Likelihood Ratio (LR) Test for the Number of Cointegration Vectors. For the trace test, impose the null hypothesis $$\text{H}_0: \rank\left(\Pi\right)=r_0 \text{ against } \text{H}_1: \rank\left(\Pi\right)>r_0$$ i.e. the number of linearly independent cointegration vectors, with the likelihood ratio statistics (trace) being $\text{LR}_{trace}(r_0)=-T\sum_{i=r_0+1}^n{\ln\left(1-\hat\lambda_i\right)}$. For the maximum eigenvalue test, conduct the following hypothesis testing:
$$\text{H}_0: \rank\left(\Pi\right)=r_0 \text{ against } \text{H}_1: \rank\left(\Pi\right)=r_0+1$$ with the likelihood ratio statistics (max) being $\text{LR}_{max}(r_0)=-T\ln\left(1-\hat\lambda_{r_0+1}\right)$.

Furthermore, with the Maximum Likelihood Estimation of the Cointegrated VECM based on the Granger's Representation Theorem, we may illustrate the procedures of the Johansen's Cointegration Test:
\begin{enumerate}
	\item Test for unit root of each time series component via Augmented Dickey-Fuller.
	\item Determine the optimal lag length for each series using an information criterion such as the Akaike Information Criterion (AIC) or the Schwarz Bayesian Criterion (SBC).
	\item Estimate a VAR model with the chosen lag length.
	\item Construct LR test for $\rank\left(\Pi\right)$.
	\begin{enumerate} 
		\item Assuming both have unit roots, then find linear approximation of relationship via OLS. Then create a series of the residuals.
		\item Test residuals for unit root via Augmented Dickey-Fuller.
	\end{enumerate}
	\item Normalise the cointegration vectors and estimate the resulting cointegrated VECM by maximum likelihood.
\end{enumerate}

\bibliographystyle{IEEEtran}
\bibliography{./reference}

\end{document}